\documentclass{amsart}
\usepackage{amsmath}
\usepackage{amsfonts}
\usepackage{amssymb}
\usepackage{amsthm}
\usepackage{amsopn}
\usepackage{amscd}
\usepackage{amsxtra}

\usepackage{mdwlist}
\usepackage{subfigure}
\usepackage{multicol}
\usepackage[utf8]{inputenc}
\usepackage{courier}

\usepackage{pstricks}
\usepackage{hyperref}

\usepackage{enumitem}
\setenumerate{label=(\roman*)}

\usepackage{listings}

\newtheorem{thm}{Theorem}[section]
\newtheorem{lemma}[thm]{Lemma}
\newtheorem{prop}[thm]{Proposition}

\newtheorem{cor}[thm]{Corollary}

\theoremstyle{definition}
\newtheorem{defn}[thm]{Definition}

\newtheorem{algorithm}[thm]{Algorithm}

\theoremstyle{remark}
\newtheorem{example}[thm]{Example}

\numberwithin{equation}{section}

 \newcommand{\NN}{\mathbb{N}}
 
\newcommand{\RR}{\mathbb{R}}
\newcommand{\CC}{\mathbb{C}}
\newcommand{\PP}{\mathbb{P}}

 \newcommand{\cB}{\mathcal{B}}

 \newcommand{\D}{\Delta}
\newcommand{\s}{\sigma} 

\newcommand{\M}{\overline{M}}
\newcommand{\LM}{\overline{L}}

\newcommand{\rlm}{\rho_{\overline{L}}}
\newcommand{\rmn}{\rho_{\overline{M}}}

\newcommand{\cix}{\mathcal{CI}(X)}
\newcommand{\ci}{\mathcal{CI}}
\newcommand{\bmov}{\overline{\mathrm{Mov}}}

\newcommand{\mori}{\overline{\mathrm{NE}}}
\newcommand{\Nef}{\mathrm{Nef}}
\newcommand{\pEff}{\overline{\mathrm{Eff}}}

\newcommand{\toricthreefold}{%
		\psset{unit=1.5cm}
                \begin{pspicture}(-3,-2)(3,2.2)
                  
                   \psset{linecolor=blue,linestyle=solid,linewidth=0.5pt}
                   \psline{-}(2,0)(2,2)
                   \psline{-}(2,2)(0,2)

                   \psset{linecolor=purple,linestyle=solid,linewidth=0.6pt}
                   \psline{-}(0,2)(1,1.5)                  
                   \psline{-}(2,0)(1, -0.5)
                   \psline{-}(2,0)(1, 1.5)
                   \psline{-}(1,1.5)(2,2)

                   \psset{linecolor=red,linestyle=solid,linewidth=0.8pt}
                   \psline{-}(1, -0.5)(1, 1.5)

                   \psset{linecolor=green,linestyle=solid,linewidth=0.3pt}
                  
                   \psset{linecolor=orange,linestyle=solid,linewidth=0.4pt}
                   \psline{-}(1,1.5)(-3, -1.5)
                   \psline{-}(1, -0.5)(-3, -1.5)
                   
                   \psset{linecolor=orange,linestyle=solid,linewidth=0.3pt}
                   \psline{-}(0,2)(-3, -1.5)
                   \psline{-}(2,2)(-3, -1.5)
                   \psline{-}(2,0)(-3, -1.5)

                  \psset{linecolor=black, linestyle=solid,linewidth=1.0pt}
                  \psline{->}(1, -0.5) 	
                  \psline{->}(1, 1.5)	
                  
                  \psset{linecolor=black, linestyle=solid,linewidth=0.8pt}
                  \psline{->}(2, 0) 	
                  \psline{->}(0,2) 	
                  \psline{->}(2,2)	
                  
                  \psset{linecolor=black, linestyle=solid, linewidth=0.4pt}
                  \psline{->}(-3, -1.5)	
                  
                  \rput(-3,-1.62){$u_0$}
                  \rput(1,-0.62){$u_1$}
                  \rput(2.0, -0.12){$u_2$}
                  \rput(0.0, 2.12){$u_3$}

                  \rput(2.0, 2.12){$u_{5}$}
                  \rput(1.1, 1.62){$u_{4}$}
                  
                  \end{pspicture}}

\author{Paul L. Larsen}

\title[Complete intersections and movable curves on $\overline{M}_{0,6}$]{Complete intersections and movable curves on the moduli space of six-pointed rational curves}
\date{\today}

\address{Humboldt-Universit\"at zu Berlin, Institut f\"ur Mathematik, 10099 Berlin, Germany}
\email{larsen@mathematik.hu-berlin.de}

\numberwithin{equation}{section}

\begin{document}
\begin{abstract}A curve on a projective variety is called movable if it belongs to an algebraic family of curves covering the variety. We consider when the cone of movable curves can be characterized without existence statements of covering families by studying the complete intersection cone on a family of blow-ups of complex projective space, including the moduli space of stable six-pointed rational curves, $\overline{M}_{0,6}$, and the permutohedral or Losev-Manin moduli space of four-pointed rational curves. Our main result is that the movable and complete intersection cones coincide for the toric members of this family, but differ for the non-toric member, $\overline{M}_{0,6}$. The proof is via an algorithm that applies in greater generality. We also give an example of a projective toric threefold for which these two cones differ.

\end{abstract}
\maketitle


\section{Introduction}
\label{secCI:introduction} 
A foundational result in the geometry of projective varieties is Kleiman's theorem \cite{MR0206009}, which states the closure of the ample cone equals the nef cone. The containment of the ample cone in the nef cone is easy to prove, and since the nef cone is by definition closed, one inclusion of cones follows. The proof of the opposite inclusion is more involved; see \cite{MR0206009}, or Section 1.4.C of \cite{MR2095471}.

By duality, Kleiman's theorem is equivalent to the equality $\mori(X)^{\vee} = \overline{\mathrm{Amp}}(X)$. It is natural to wonder which other cones of divisor and curve classes fit into a Kleiman-type duality. For the pseudoeffective cone of divisor classes, it is not difficult to see that dual cone $\pEff(X)^{\vee}$ contains the closure of the cone of movable curve classes, where a reduced, irreducible curve $C$ is called a \emph{movable curve} if $C= C_{t_0}$ belongs to an algebraic family $(C_t)_{t \in S}$ covering $X$. To see this inclusion, 
let $D$ be an effective prime divisor, and let $C$ be a movable curve. Since the support of $D$ is a codimension one subvariety, there must exist an irreducible curve $C'$ in the covering family containing $C$ such that $C'$ is not contained in the support of $D$, hence $C' \cdot D \geq 0$. Since algebraic equivalence is finer than numerical equivalence, it follows that $C \cdot D \geq 0$. The other inclusion was proved in 2004 by Boucksom, Demailly, P\v aun, and Peternell in \cite{bdpp}, where they also give an alternative characterization of the cone of movable curve classes:
\begin{defn}
\label{defCI:movable}
Let $\mu: X' \to X$ be a projective, birational morphism. A class $\gamma \in \mori(X)$ is called \emph{movable} if there exists a representative one-cycle $C$ and ample divisors $A_1, \ldots, A_{\dim(X) -1}$ on $X'$ such that
\begin{equation}
\mu_*(A_1 \cdot \ldots \cdot A_{\dim(X)-1}) = C \nonumber.
\end{equation}
The closure of the cone generated by movable classes in $\mori(X)$ is called the \emph{movable cone}, and is denoted $\bmov(X)$.
\end{defn}

Both formulations involve non-trivial existence statements: in the first, to see that a curve $C$ is movable, we must prove the existence of a covering family to which it belongs, and in the second, we require knowledge about all projective, birational morphisms to the variety $X$. If, however, we consider only the identity morphism, we obtain a subcone of $\bmov(X)$ called the \emph{complete intersection cone}:
\begin{defn}
\label{defCI:ci}
The \emph{complete intersection} cone of $X$, denoted $\ci(X)$, is the closed cone generated by the classes of all smooth curves obtained as an intersection of $\dim(X)-1$ ample divisors on $X$.
\end{defn}

The aim of this paper is to investigate when these cones of curve classes do and do not coincide for two natural testing grounds: moduli spaces of curves and toric varieties. Were there actual equality $\ci(X) = \bmov(X)$, then we could characterize movable curves without having to first classify all birational morphisms to $X$. A disadvantage of working with the complete intersection cone, however, is the combinatorial complexity of $\ci(X)$, especially when the nef cone of $X$ has a large number of extremal rays. 
\begin{example}
\label{exCI:CIvsMovSurface}
Let $X$ be a smooth projective surface. Then one-cycles and divisors coincide, so $\pEff(X)^{\vee} = \Nef(X) = \ci(X)$, where the second equality follows from Kleiman's theorem, since by definition $\ci(X)$ is the closure of the ample cone. 
\end{example}
\begin{example}
\label{exCI:CIvsMovPn}
Let $X= \PP^n$, and let $H \subseteq \PP^n$ be a hyperplane, and let $\ell \subseteq \PP^n$ be a line. Then $N^1(\PP^n)_\RR = \langle [H]\rangle$ and $\pEff(\PP^n) = \Nef(\PP^n) = \langle [H] \rangle_{\geq 0}$, while $N_1(\PP^n)_\RR = \langle [\ell]\rangle$, and $\pEff(\PP^n)^\vee = \bmov(\PP^n) = \langle [\ell] \rangle_{\geq 0}=  \langle [H]^{n-1} \rangle_{\geq 0}$, hence $\ci(\PP^n) = \bmov(\PP^n)$.
\end{example}
Peternell has calculated an example of a smooth projective threefold for which the containment of the complete intersection cone in the movable cone is strict \cite{pMov}, but one can ask if there are natural families of varieties for which these cones coincide. Two obvious testing grounds are toric varieties and moduli spaces of stable pointed rational curves, since the intersection theory on these varieties is well-understood. A connection between these two families is the Kapranov blow-up construction. In \cite{MR1237834}, $\M_{0,n}$ is constructed by a series of toric blow-ups of $\PP^{n-3}$, culminating in the permutohedral or Losev-Manin moduli space $\LM_{n-2}$, followed by (for $n \geq 5$) additional blow-ups along non-torus-invariant centers. 

Example \ref{exCI:CIvsMovPn} can be taken as the base case of a progression of varieties obtained by successive Kapranov-like blow-ups. More specifically, setting $X_0 = \PP^3$, the next variety we take to be the blow-up of $\PP^3$ at a general point, labeling the resulting variety $X_1$. We define $X_2$ to be the blow-up of $\PP^3$ along two general points, and the proper transform of the line spanned by the points. In general, for $1 \leq r \leq  5$, we blow-up $r$ points of $\PP^3$ in general linear position, and then the proper transforms of the $\binom{r}{2}$ lines generated by the $r$ points. For $r \leq 4$, the centers of the blow-ups can be chosen to be torus-invariant. Then $X_4$ is the permutohedral space $\LM_4$, while $X_5$ is $\M_{0,6}$. The complete intersection and movable cones of the first few varieties $X_r$ can be computed easily to show that these cones coincide, but there is little reason to expect this equality of cones to be preserved under increasing blow-ups. 

The main result of this paper is the following: 
\begin{thm}
\label{thmCI:CIvsNef}
There is a strict inclusion $\ci(\M_{0,6}) \subsetneq \bmov(\M_{0,6})$, while for the toric varieties $X_r$, $1 \leq r \leq 4$, equality holds: $\ci(X_r) = \bmov(X_r)$. 
\end{thm}
\noindent In other words, the containment of these cones becomes strict when we leave the toric world in the Kapranov construction of $\M_{0,6}$.

We prove this theorem by reinterpreting the complete intersection cone in combinatorial terms (see Definition \ref{defCI:Nefnm1} and Lemma \ref{lemmaCI:cinef2}). Since the nef and pseudoeffective cones of $\M_{0,6}$ and $\LM_4$ are finitely generated, it follows by this reinterpretation that equality of the moving and complete intersection cones can be tested by an algorithm that requires as input the extremal rays of the nef and effective cones of divisors, plus intersection products of divisors (see Section \ref{secCI:ci}). 

That the complete intersection and movable cones coincide for the toric blow-ups of Theorem \ref{thmCI:CIvsNef} might give hope that these cones coincide for smooth projective toric varieties. It turns out, however, that even for a toric blow-up of projective space the complete intersection cone need not equal the movable cone. In Example \ref{exCI:toricthreefold}, we produce such a toric variety.

The remainder of this paper is organized as follows. We begin with some generalities on the pseudoeffective and nef cones of divisors, as well as the the closed cones of curves, for $\M_{0,n}$ and the other blow-ups $X_r$ in Section \ref{secCI:background}. In Section \ref{secCI:ci}, we establish a combinatorial definition of the complete intersection cone, and describe the algorithm used to prove Theorem \ref{thmCI:CIvsNef}. We also show that extremal movable curve classes of the toric variety $\LM_{n-2}$ pull back to extremal classes in $\M_{0,n}$. Section \ref{secCI:intTh} contains proofs for intersection calculations used for our algorithm.

\textbf{Acknowledgments}: I would first like to thank Gavril Farkas for suggesting this problem, and for his help throughout. This project has benefitted greatly from conversations with Nathan Ilten, Sam Payne, and Thomas Peternell. I would also like to thank Klaus Altmann and Angela Gibney for comments on an earlier version of this paper.
\section{Definitions and background}
\label{secCI:background}
In this section we give definitions from intersection theory on a complex projective variety before focusing on the particular examples of toric varieties and the moduli space of stable pointed rational curve, $\M_{0,n}$. We refer to \cite{MR2095471}, \cite{MR1644323}, and Appendix A of \cite{MR0463157} for the basics of intersection theory, and \cite{MR1034665} and \cite{larsenThesis} for background and examples involving $\M_{0,n}$.

Let $X$ be a smooth complex projective variety, with $\mathrm{Div}_{\RR}(X)$ denoting the space of $\RR$-linear formal sums of algebraic hypersurfaces on $X$ (called \emph{$\RR$-divisors} on $X$). There is a well-defined intersection pairing between $\RR$-divisors and $\RR$-linear formal sums of algebraic curves on $X$ (called \emph{one-cycles}), which we denote by ``$\cdot$''.
\begin{defn}
\label{def:NS}
Two divisors $D_1, D_2 \in \mathrm{Div}_{\RR}(X)$ are said to be \emph{numerically equivalent} if for all algebraic curves $C \subseteq X$, $D_1 \cdot C = D_2 \cdot C$.
\end{defn}
\noindent We thus obtain an equivalence relation on $\mathrm{Div}_{\RR}(X)$, and denote the numerical equivalence class of a divisor $D$ by $[D]$.

\begin{defn}
The \emph{N\'eron-Severi} space of $X$ is defined as
\begin{equation*}
N^1(X)_{\RR} = \{ [D] : D \in \mathrm{Div}_{\RR}(X) \}, 
\end{equation*}
and the dual vector space induced by the intersection product is denoted $N_1(X)_\RR$.
\end{defn}

\noindent We will also denote the numerical class of a one-cycle $C$ as $[C] \in N_1(X)_\RR$. A key fact for what follows is that $N^1(X)_\RR$ and $N_1(X)_\RR$ are finite-dimensional $\RR$-vector spaces. 

\begin{defn}
\label{def:effDiv}
The \emph{pseudoeffective cone} of divisors, written $\pEff(X)$, is the closed subcone of $N^1(X)_\RR$ generated by classes of effective divisors. Explicitly, $\pEff(X)$ is the closure in $N^1(X)_\RR$ of
\begin{equation*}
\mathrm{Eff}(X) = \big\{ \sum d_i [D_i]: d_i \geq 0, D_i \textrm{ an effective divisor on }X \big\}.
\end{equation*}
\end{defn}

\noindent The analogous cone in $N_1(X)_\RR$ is called the \emph{closed} (or \emph{Mori}) \emph{cone of  curves}:
 \begin{defn}
Define $\mori(X)$ as the closed subcone of $N_1(X)_\RR$ generated by classes of algebraic curves, that is, the closure in $N_1(X)_\RR$ of
\begin{equation*}
\mathrm{NE}(X) = \big\{ \sum c_i [C_i]: c_i \geq 0, C_i \subseteq X \textrm{ an algebraic curve} \big\}.
\end{equation*}
\end{defn}

\begin{defn}
The cone of \emph{nef divisors} on $X$ is
\begin{equation}
\Nef(X) = \{[D] \in N^1(X)_\RR: D \cdot C \geq 0 \textrm{ for all }[C] \in \mori(X)\} = (\mori(X))^{\vee}.\nonumber
\end{equation}
\end{defn}

We will sometimes reduce intersection properties on $\M_{0,n}$ to intersections on a more amenable variety via the \emph{projection formula}. We use this result for numerical equivalence classes on non-singular varieties, where the formula takes on a particularly simple form (see \cite{MR1644323}, Proposition 8(c), noting that rational equivalence is finer than numerical equivalence). We only give the formula for divisors and one-cycles, though the statement holds for all pairs of complementary dimensional subvarieties.
\begin{lemma}[Projection formula]
\label{lemmaI:projection}
 Let $f: X \to Y$ be a proper morphism of non-singular varieties. For $\delta \in N^1(X)_\RR$ and $\gamma \in N_1(Y)_\RR$,
\begin{equation*}
f_*(\delta \cdot f^*\gamma) = f_*(\delta) \cdot \gamma.
\end{equation*}
\end{lemma}

We next give a short description of the moduli space $\M_{0,n}$. Set theoretically, it is defined as follows:

\begin{defn}
For $n \in \NN$, $n \geq 3$, the elements of $\M_{0,n}$ are equivalence classes of
\begin{equation*}
\{ (C, p_1, \ldots, p_n): C \textrm{ a tree of } \PP^1s, p_i \in C \textrm{ distinct} \}
\end{equation*}
such that
\begin{enumerate}
\item all marked points $p_i$ are distinct from the nodes of $C$,
\item each irreducible component $C$ contains at least three marked or singular points,
\item two marked curves $(C, p_1, \ldots, p_n)$ and $(C', q_1, \ldots, q_n)$ are equivalent if there is an isomorphism $\phi: C \to C'$ with $\phi(p_i) = q_i$ for all $i$.
\end{enumerate}
\end{defn}
The resulting moduli space is a smooth projective variety \cite{MR702953}, and can be realized as a sequence of blow-ups of $\PP^{n-3}$ along linear centers \cite{MR1237834}. This blow-up construction will feature heavily in the remainder, and although it corresponds to that of \cite{MR1237834}, the ordering of the blow-ups differs slightly from that given in subsequent literature. First pick $n-2$ general points in $\PP^{n-3}$; via a projective transformation, we may choose $x_1 = [1, 0, \ldots, 0]$, $\ldots$, $x_{n-2} = [0, \ldots, 0, 1] \in \PP^{n-3}$. Note that these points are invariant under the action of the torus $(\CC^*)^{n-3}$. First blow up $\PP^{n-3}$ iteratively along $x_1, \ldots, x_{n-2}$, then along the proper transforms of the lines spanned by pairs of $x_1, \ldots, x_{n-2}$, and continue blowing up proper transforms of linear subspaces spanned by these points until all codimension two subspaces have been blown up. Since all blow-up centers were torus invariant, the result is a smooth projective toric variety.

\begin{defn}
The variety obtained from the above blow-ups of $\PP^{n-3}$ is the \emph{permutohedral} or \emph{Losev-Manin moduli space} $\LM_{n-2}$.
\end{defn}

\noindent The name \emph{permutohedral space} is given in \cite{MR1237834} since the corresponding polytope in a permutahedron, while the second variant reflects the modular interpretation given in \cite{MR1786500}. More on these varieties can be found in \cite{blume1} and \cite{pllCoxRelns}.

To obtain $\M_{0,n}$ from $\LM_{n-2}$ we first blow up the point $x_{n-1} = [1, \ldots, 1]$, and then---in order of increasing dimension---all remaining proper transforms of linear subspaces spanned by $x_1, \ldots, x_{n-1}$. These blow-up constructions also give natural bases for $N_1(\LM_{n-2})_\RR$ and $N_1(\M_{0,n})_\RR$, which we will call the \emph{Kapranov basis}: 
\begin{itemize}
\item For $\LM_{n-2}$, denote by $[H]$ the pull-back of the hyperplane class on $\PP^{n-3}$, and by $[E_J]$ with $J \subseteq \{1, \ldots, n-2\}$, $1 \leq |J| \leq n-4$ the proper transforms of the exceptional divisors resulting from blowing up a linear subspace spanned by the collection of $\{x_1, \ldots, x_{n-2}\}$ indexed by $J$.

\item For $\M_{0,n}$, we abuse notation by again writing $[H]$ for the pull-back of the hyperplane class, and as above $[E_J]$ for the proper transforms of exceptional divisors, where now $J \subseteq \{1, \ldots, n-1\}$.
\end{itemize}

We will be especially concerned with two collections of divisor and curve classes on $\M_{0,n}$, namely those of \emph{boundary divisors} and \emph{F-curves}.
\begin{defn}
For $J  \subseteq \{1, \ldots,n \}$ with $2 \leq |J| \leq n-2$, the \emph{boundary divisor} $\D_J$ is the locus of elements in $\M_{0,n}$ whose underlying curve can be decomposed into two components $C = C' \cup C''$ such that the marked points on $C'$ are indexed by $J$ and those of $C''$ are indexed by $J^c$.
\end{defn}
\noindent We will tacitly identify $\D_J$ and $\D_{J^c}$. The boundary divisors form the codimension one constituents of a stratification of $\M_{0,n}$ by dual-graph. The dimension one elements of this stratification are known as \emph{F-curves}. The numerical equivalence class of an F-curve is uniquely determine by a partition of $\{1, \ldots, n\}$ into four subsets, $\{\mu_1, \mu_2, \mu_3, \mu_4\}$. 
To go from F-curves to such partitions, let $(C, p_1, \ldots, p_n) \in \M_{0,n}$ be a generic element of an F-curve. Then $C$ can be decomposed as $C = C_{spine} \cup C_1 \cup C_2 \cup C_3 \cup C_4$, where $C_{spine}$ is a $\PP^1$ with four special points, and the marked points indexed by $\mu_i$ are located on $C_i$ for all $i$. That this partition uniquely determines the numerical class of an F-curve results from the following intersection pairings, proved in \cite{km}.
\begin{prop}
\label{propI:FcurvePartition}
Let $F_\mu$ be a one stratum, with corresponding partition $\mu = (\mu_1, \mu_2, \mu_3, \mu_4)$. For any boundary divisor $\Delta_J$,
\begin{displaymath}
F_\mu \cdot \Delta_J = 
\left\{ \begin{array}{rl}
-1 & \text{if $J$ or $J^c$ equals $\mu_i$ for some $i$} \\
1 & \text{if $J = \mu_i \cup \mu_j$ for some $i \neq j$},\\
0 & \text{otherwise}.
\end{array} \right.
\end{displaymath}

\end{prop}
\noindent Since classes of boundary divisors generate $N^1(\M_{0,n})_\RR$, these intersection numbers uniquely determine the class of $F$.

We will require one final fact about intersection theory on $\M_{0,n}$ relating boundary divisors and elements of the Kapranov basis:
\begin{equation}\label{eq:dictionary}
\begin{split}
\Delta_{J \cup \{n\}} &= E_J, \textrm { if } 1 \leq |J| \leq n-4, \\
\big[\Delta_{J \cup \{n\}}\big] & = [H] - \bigg(\sum_{J' \subsetneq J}[E_{J'}] \bigg), \textrm{ if } |J| = n-3.
\end{split}
\end{equation}

Except for small values of $n$, little is known about the pseudoeffective cone of divisors or the closed cone of curves for $\M_{0,n}$. The pseudoeffective cone of $\M_{0,n}$ (and hence, by duality, the movable cone of curve classes) is known to be finitely generated only for $n \leq 6$ \cite{MR1941624,MR2491903}, while finite-generation of the closed cone of curve classes of $\M_{0,n}$ (and hence, by duality, the cone of nef divisor classes) has been proven for $n \leq 7$ \cite{km, Larsen16112011}. 

For $\M_{0,6}$, the closed cone of curve classes is generated by classes of F-curves, while the pseudoeffective cone of $\M_{0,6}$ is generated by the boundary divisors $\Delta_J$, and the \emph{Keel-Vermeire} divisors \cite{MR1882122}. In the Kapranov blow-up description of $\M_{0,6}$, Keel-Vermeire divisors are the pull-backs under the blow-up morphism $t_6: \M_{0,6} \to \PP^3$ of the unique quadric surface containing points $p_1$, $\ldots$, $p_5$, and the lines $l_{ac}$, $l_{ad}$, $l_{bc}$ and $l_{bd}$. Taking $(a,b,c,d) = (1,2,3,4)$, the Keel-Vermeire divisor $Q_{(12)(34)(56)}$ has numerical class
\begin{equation}
\label{eqCI:KV}
[Q_{(12)(34)(56)}] = 2[H] - \sum_{i=1}^5[E_i] - [E_{13}] - [E_{14}] - [E_{23}] - [E_{24}].
\end{equation}
\noindent The remaining fourteen Keel-Vermeire divisors arise by varying the indexing product of two cycles (here meant in terms of symmetric groups, not algebraic cycles).

For toric varieties, each of the cones described above is finitely generated, and admits an explicit description (sometimes more than one) in combinatorial terms. The starting point for understanding the various cones of a toric variety $X_\Sigma$ of dimension $d$ is the Orbit-Cone correspondence (see  for example \cite{MR2810322}, §3.2 and §6.3). Recalling that $\Sigma(k)$ denotes the $k$-dimensional cones of the fan $\Sigma$, and $V(\s)$ is the codimension $k$ subvariety corresponding to $\s \in \Sigma(k)$, we have the following descriptions of $\pEff(X_\Sigma)$ and $\mori(X_\Sigma)$:
\begin{prop} For a complete toric variety $X_\Sigma$,
\label{propCI:pseffToric}
\begin{equation}
\pEff(X_\Sigma) = \langle [V(\rho)]: \rho \in \Sigma(1) \rangle_{\geq 0}, \nonumber
\end{equation}
while
\begin{equation}
\mori(X_\Sigma) = \langle [V(\tau)]: \tau \in \Sigma(d-1) \rangle_{\geq 0}. \nonumber
\end{equation}
\end{prop}
\noindent The duals of these cones, $\bmov(X_\Sigma)$ and $\Nef(X_\Sigma)$, can be calculated by the combinatorics of the defining fans. Example calculations appear in Example \ref{exCI:toricthreefold}.
\section{Comparing complete intersection and movable curve classes}
\label{secCI:ci}
The proof of Theorem \ref{thmCI:CIvsNef} involves comparing
intersections of pairs of nef divisor classes with movable classes. To begin this section, we give defining inequalities for the nef
and movable cones, and calculate intersections of pairs of divisors on
the varieties $X_r$ (defined below), thus providing the input data for
the algorithmic proof of Theorem \ref{thmCI:CIvsNef}. Proofs are given in Section \ref{secCI:intTh}.

\begin{defn}
\label{defnCI:kapBases}
Let $X_r$ be the composition of the blow-ups of $r$ general points in
$\PP^3$, $1 \leq r \leq 5$, followed by the blow-ups of the
proper-transforms of the $\binom{r}{2}$ lines of $\PP^3$ spanned by
the $r$ points.
\end{defn}
Note that $X_4 = \LM_4$ and $X_5 = \M_{0,6}$. Let $H$ be the pullback of a general hyperplane, let $E_1, \ldots, E_r$ be the exceptional divisors obtained by blowing up the points, and let $E_{12}, \ldots, E_{r-1 \, r}$ be the proper transforms of the exceptional divisors obtained by blowing up the lines.

As in Section \ref{secCI:background}, the \emph{Kapranov basis} of $N^1(X_r)_\RR$ is
\begin{equation*}
\{[H], [E_1], \ldots, [E_r], [E_{12}], \ldots, [E_{r-1 \, r}] \},
\end{equation*}
and the \emph{dual Kapranov basis} of $N_1(X_r)_\RR$ is denoted
\begin{equation*}
\{[H]^{\vee}, [E_1]^{\vee}, \ldots, [E_r]^{\vee}, [E_{12}]^{\vee}, \ldots, [E_{r-1 \, r}]^{\vee} \}.
\end{equation*}

\noindent Again, we will abuse notation and not distinguish notationally
among the Kapranov bases from the different $X_r$. 

To characterize the nef cones of the $X_r$, we adopt the notational
convention for the defining inequalities that a coefficient $d_{ij}$ is set to zero if the indices
are impossible for a given inequality. For example, if $r=1$, then all
$d_{ij}$ appearing below are taken to be 0, and if $r=2$, the final inequality below reads $d_h + d_i \geq 0$ for $i \in \{1,2\}$. We also identify
$d_{ij}$ and $d_{ji}$.
\begin{prop}
\label{lemmaCI:nefCone} 
Let $[D] = d_h [H] + \sum_{i=1}^r d_i [E_i] + \sum_{1 \leq j<k \leq r} d_{jk}[E_{jk}]$ be an arbitrary divisor class in $X_r$. The cone of nef divisors is determined by the inequalities

\begin{equation*}
 \begin{cases}
 -d_{ij} \geq 0, &\text{for } 1 \leq i < j \leq r,\\
 d_h + d_i + d_j - d_{ij} \geq 0, &\text{for } 1 \leq i <  j \leq r,\\
 -d_i + d_{ij} + d_{ik} \geq 0, &\text{for } 1 \leq i,j,k \leq r, i \notin\{j,k\}, \\
 d_h + d_i + d_{jk} + d_{lm} \geq 0, &\text{for } \{i,j,k,l,m\} = \{1, \ldots, r\}.
 \end{cases}
\end{equation*}

\end{prop}

To determine the movable cone of curves, $\bmov(\M_{0,6})$, we consider intersections of one-cycles with the generators of $\pEff(\M_{0,6})$, that is, with all boundary divisor classes $[\Delta_J]$ and the fifteen Keel-Vermeire divisor classes $[Q_{(ab)(cd)(e6)}]$. We apply an analogous convention used to characterize $\Nef(X_r)$ to the terms $c_i$ and $c_{jk}$ in the inequalities below.

\begin{prop}
\label{lemmaCI:nefCurvesCone}
Let $[C] = c_h [H]^\vee + \sum_{i=1}^r c_i [E_i]^\vee + \sum_{1 \leq j < k \leq r} c_{jk} [E_{jk}]^\vee$ be a one cycle class in $N_1(\M_{0,6})$. The cone of movable curve classes in $X_r$ is determined by the inequalities
\begin{equation*}
 \begin{cases}
 c_{i} \geq 0, &\text{for }i=1, \ldots, r, \\
 c_{jk} \geq 0, &\text{for } 1\leq j < k \leq r,  \\
 c_h - c_i - c_j  - c_k - c_{ij} - c_{ik} - c_{jk} \geq 0, &\text{for }1 \leq i< j < k \leq r, \\
 2 c_h - \sum_{i=1}^5 c_i - c_{jl} - c_{kl} - c_{jm} - c_{km} \geq 0, &\text{for } j,k,l,m \in \{1, \ldots, r\}\text{ distinct}.
 \end{cases}
\end{equation*}
\end{prop}
\noindent Note that for each inequality of the last type follows from
inequalities of the first three types when $r \leq 4$.

We will give an alternative definition of the complete intersection
cone in Lemma \ref{lemmaCI:cinef2}  involving intersections of nef divisors, so we next write the remaining intersections of elements of the Kapranov basis for $N^1(\M_{0,6})_{\RR}$ in terms of the dual basis.
\begin{prop}
\label{lemmaCI:doubleInt}
The intersections of elements of the Kapranov basis for $X_r$, in terms of the dual basis, are, for distinct $i,j,k,l \in \{1, \ldots, r\}$,

\begin{equation} \label{eq:intZero}
  \; \; \qquad
0 = [H] \cdot [E_i] =  
[E_i] \cdot [E_j] = 
[E_i] \cdot [E_{jk}] =
[E_{ij}] \cdot [E_{kl}] =
[E_{ij}] \cdot [E_{ik}],
\end{equation}
\begin{equation} \label{eq:intDual}
\begin{split}
[H]^2& = [H]^{\vee}, \;
[H]\cdot [E_{jk}]  = [E_j] \cdot [E_{jk}] = [E_k] \cdot [E_{jk}] = -[E_{jk}]^\vee,\\
[E_i]^2&  = [E_i]^{\vee}, \;
[E_{jk}]^2 = 2 [E_{jk}]^\vee - [H]^\vee - [E_j]^\vee - [E_k]^\vee.
\end{split}
\end{equation}
\end{prop}

Now we turn to the algorithm used to prove Theorem \ref{thmCI:CIvsNef}. We begin with a recasting of the complete intersection cone of a projective variety $X$ with a finitely generated nef cone.
\begin{defn}
\label{defCI:Nefnm1}
For $X$ a smooth projective variety of dimension $d$ with a finitely generated nef cone, define $(\Nef(X))^{d-1} \subseteq N_1(X)_\RR$ as
\begin{equation}
 (\Nef(X))^{d-1} = \langle  [N_1] \cdot \ldots \cdot [N_{d-1}]: \textrm{ each }[N_i] \textrm{ an extremal ray of }\Nef(X) \rangle_{\geq 0} \nonumber
\end{equation}
\end{defn}
\noindent Note that finite generation of $\Nef(X)$ implies that $(\Nef(X))^{d-1}$ is a closed cone.

\begin{lemma}
\label{lemmaCI:cinef2}
The cones $(\Nef(X))^{d-1}$ and $\cix $ are equal.
\end{lemma}

\begin{proof}
 To see that $(\Nef(X))^{d-1} \subseteq \cix$, note first that every
 nef divisor is a limit of ample divisors (see \cite{MR2095471},
 Section 1.4). Since $\cix$ is a closed cone, multilinearity and continuity of the
 intersection product (\cite{MR2095471}, Section 1.1) imply the first
 inclusion. The reverse inclusion is obvious.
\end{proof}

\begin{cor}
\label{corCI:algorithm}
If the nef cone of $X$ is finitely generated, the cones $\ci(X)$ and $\bmov(X)$ coincide if and only if every extremal ray of $\bmov(X)$ is a non-trivial multiple of a generator of $(\Nef(X))^{d-1}$.
\end{cor}

This corollary leads directly to an algorithm to test the equality of
$\cix$ and $\bmov(X)$ when $\Nef(X)$ and $\bmov(X)$ are finitely
generated, and all necessary intersections are known. We describe the
algorithm in detail for a projective three-fold; the extension of the
algorithm to higher-dimensional varieties will be obvious. 

\begin{algorithm}
Determine if $\cix = \bmov(X)$ for a smooth projective threefold.
\begin{description}
\item[Input]Extremal rays $\gamma_1, \ldots, \gamma_s$
of $\bmov(X)$ with respect to the basis $\cB_1$ of $N_1(X)_\RR$; 
extremal rays $\eta_1, \ldots, \eta_r$ of
$\Nef(X)$ with respect to the dual basis (with respect to the
intersection product) $\cB^1$ of $N^1(X)_\RR$;
intersection products of pairs from $\cB^1$ in the basis $\cB_1$.
\item[Output]Extremal rays $\gamma_{i_1}, \ldots, \gamma_{i_t}$ of
$\bmov(X)$ not in $\cix$ to a file \verb+NotEq+.
\end{description}

\begin{enumerate}
\item[1.a.] Read in $\gamma_1$.
\item[1.b.] For each pair $1 \leq i \leq j \leq r$,  calculate $\eta_i
  \cdot \eta_j$ with respect to $\cB_1$. If $\eta_i \cdot \eta_j$ is a
  non-trivial multiple of $\gamma_1$, continue to the next pair $1
  \leq i' \leq j' \leq r$. Otherwise output $\gamma_1$ to \verb+NotEq+ and
  continue to the next pair. 
\item[2.a.] Read in $\gamma_2$.
\item[2.b.] \ldots \\
\item[$s$.a.]Read in $\gamma_s$.
\item[$s$.b.] For each pair $1 \leq i \leq j \leq r$,  calculate $\eta_i
  \cdot \eta_j$ with respect to $\cB_1$. If $\eta_i \cdot \eta_j$ is a
  non-trivial multiple of $\gamma_s$, continue to the next pair $1
  \leq i' \leq j' \leq r$. Else output $\gamma_s$ to \verb+NotEq+ and
  continue to the next pair. 
\end{enumerate}
\end{algorithm}

The cones $\cix$ and $\bmov(X)$ are equal precisety when the file
\verb+NotEq+ is empty after running the algorithm. Note that the
second step for each ray $\gamma_i$ involves recalculating all generators for $(\Nef(X))^2$. This apparent inefficiency is in practice preferable to storing every generator $\eta_{t_1} \cdot \eta_{t_2}$ in an array due to memory requirements and the computational time required to access elements in this array.

An implementation of the algorithm as a C++ program for $X_r$, $r=2, 
\ldots, 5$, is available at \href{http://www.math.hu-berlin/~larsen/papers.html}{www.math.hu-berlin/$\sim$larsen/papers.html} (for
$X_1$ the algorithm is easy to implement by hand). We obtain
enumerations of the extremal rays of $\Nef(X)$ and $\bmov(X)$ by
inputting the inequalities from Propositions \ref{lemmaCI:nefCone} and
\ref{lemmaCI:nefCurvesCone} into software that implements
Fourier-Motzkin elimination, such as \verb+PORTA+ \cite{porta}. These \verb+PORTA+
files are also available at \href{http://www.math.hu-berlin/~larsen/papers.html}{www.math.hu-berlin/$\sim$larsen/papers.html}.  Intersections of pairs of nef divisors are calculated according to Proposition \ref{lemmaCI:doubleInt}. Running these programs yields:
\begin{cor}
\label{corCI:CIvsNef}
There is a strict inclusion $\ci(\M_{0,6}) \subsetneq \bmov(\M_{0,6})$, while $\ci(X_r) = \bmov(X_r)$ for $r=1, \ldots, 4$.
\end{cor}
\noindent A mostly by-hand implementation of the algorithm appears in
Example \ref{exCI:toricthreefold}.
\begin{cor}
Extremal rays of $\bmov(\M_{0,6})$ not contained in $\ci(\M_{0,6})$
intersect the canonical class of $\M_{0,6}$ negatively.
\end{cor}
\begin{proof}
The canonical class of $\M_{0,6}$ is
\begin{equation*}
[K_{\M_{0,6}}] = -4[H] + 2 \sum_{i=1}^5 [E_i] + \sum_{1 \leq j \leq k
  \leq 5} [E_{jk}].
\end{equation*}
Inspection of the file \verb+NotEq+ for $r=5$ then yields the result.
\end{proof}

An obvious question to ask is whether the complete intersection and
movable cones coincide for all smooth projective toric varieties. We
next give an example of a toric blow-up of $\PP^3$ for which the complete intersection cone is strictly contained in the movable cone.
\begin{figure}
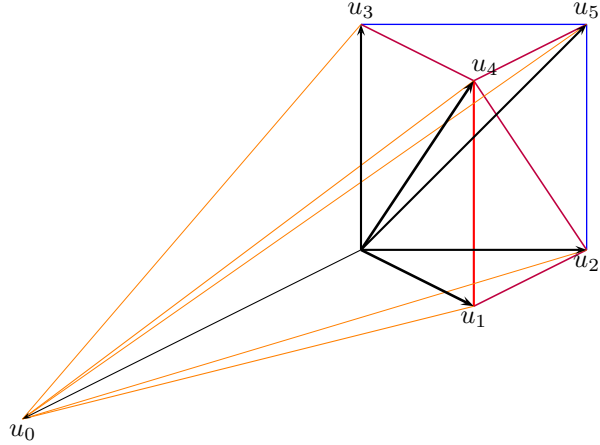

\centering{
\toricthreefold
}
\caption{Fan of the toric variety from Example \ref{exCI:toricthreefold}}
\label{figCI:toricthreefold}
\end{figure}

\begin{example}
\label{exCI:toricthreefold}
Let $Y_2$ be the toric variety obtained by blowing up $\PP^3$ first in
the line $V(z_1, z_3)$, followed by the blow-up of the
proper-transform of the line $V(z_1, z_2)$, where $\CC[z_0, z_1, z_2,
z_3]$ is the homogeneous coordinate ring of $\PP^3$. The variety $Y_2$ is smooth and projective, and its fan $\Sigma$ is depicted in Figure \ref{figCI:toricthreefold}, with the other segments indicating the two-faces of the fan. The standard facts about divisor classes and intersection theory on toric varieties reviewed and used in this example can be found in \cite{MR2810322}, Chapters 4 and 6.

The primitive generators $u_0, \ldots, u_5$ of the fan of $Y_2$ are, respectively,
\begin{equation*}
\left(\begin{array}{r} -1 \\ -1  \\ -1 \end{array} \right),
\left(\begin{array}{r} 1 \\ 0 \\ 0 \end{array} \right),
\left(\begin{array}{r} 0 \\ 1 \\ 0 \end{array} \right),
\left(\begin{array}{r} 0 \\ 0 \\ 1 \end{array} \right),
\left(\begin{array}{r} 1 \\ 0 \\ 1\end{array} \right),
\left(\begin{array}{r} 0 \\ 1 \\ 1\end{array} \right).
\end{equation*}
As usual, we label the torus-invariant divisors via the (primitive
generators of the) rays of $\Sigma(1)$, so $N^1(Y_2)_\RR$ is generated
by the classes of the divisors $D_{0}, \ldots, D_5$. The relations among the classes of divisors $D_i$ are generated by
\begin{align*}
&[D_1] + [D_4] - [D_0] = 0, \\
&[D_2] + [D_5] - [D_0] = 0, \\
&[D_3] + [D_4] + [D_5] - [D_0] = 0. 
\end{align*}

 We now perform the calculations of Algorithm \ref{corCI:algorithm} to compare $\ci(Y_2)$ and $\bmov(Y_2)$.  It is clear from the relations in $N^1(Y_2)_\RR$ that the pseudoeffective cone is
\begin{equation}
\pEff(Y_2) = \langle [D_3], [D_4], [D_5] \rangle_{\geq 0}. \nonumber
\end{equation}
We therefore choose the basis for $N^1(Y_2)_\RR$ consisting of the classes of $D_3$, $D_4$, and $D_5$, while for $N_1(Y_2)_\RR$ we take the corresponding dual basis. It follows that the movable cone is
\begin{equation}
\bmov(Y_2) = \langle [D_3]^{\vee}, [D_4]^{\vee}, [D_5]^{\vee} \rangle_{\geq 0}, \nonumber
\end{equation}
or, in coordinates, the non-negative orthant of $\RR^3$.

As noted in Proposition \ref{propCI:pseffToric}, the closed cone of curves of $Y_2$ is generated by classes of orbit closures $V(\tau)$, $\tau \in \Sigma(2)$, and we will label them as $V(\tau) = C_{i,j}$, where $i$ and $j$ index the rays generating $\tau$.

Writing an arbitrary divisor class as $[D] = d_3 [D_3] + d_4 [D_4] + d_5 [D_5]$, the nef cone of $Y_2$ is defined by the inequalities
\begin{align*}
   d_4 & =[C_{0,1}] \cdot [D] = [C_{1,2}] \cdot [D] = [C_{1,4}] \cdot [D]= [C_{2,5}] \cdot [D] \geq 0, \\
 d_5 &=[ C_{0,2}] \cdot[ D]  \geq 0, \\
 -d_4 + d_5 & = [C_{2,4}] \cdot[ D] \geq 0, \\
-d_3 + d_4 + d_5 & =[ C_{0,3}] \cdot [D] = [C_{3,4}] \cdot[ D] = [C_{3,5}] \cdot[ D] \geq 0, \\
 d_3 - d_5 & = [C_{0,5}]\cdot[ D] =[ C_{4,5}] \cdot[ D]  \geq 0, \\
 d_3 - d_4 & = [C_{0,4}] \cdot[ D] \geq 0.
\end{align*}
These intersections can be calculated via the geometry of the fan of $Y_2$ as in \cite{MR2810322}, Section 6.3.

For example, to obtain the fourth inequality above, we intersect $D$ with the curve $C_{0,3}$. Setting $C_{0,3} = V(\tau)$, with $\tau = \langle u_0, u_3 \rangle_{\geq 0}$, note that $\tau$ is contained precisely in the full-dimensional cones $\langle u_4, u_0, u_3 \rangle_{\geq 0}$ and $\langle u_0, u_3, u_5 \rangle_{\geq 0}$. We obtain from the coefficients of the linear dependence relation
\begin{equation*}
(1)\left(
\begin{array}{r}
	1  \\
	0  \\
	1 
	\end{array} \right)
+(1)\left(
\begin{array}{r}
-1 \\
-1 \\
-1
\end{array}
\right)
+(-1)
\left(
\begin{array}{r}
0 \\
0 \\
1
\end{array}
\right)
+
(1)\left(
\begin{array}{r}
0 \\
1 \\
1
\end{array}
\right)
=
\left(
\begin{array}{r}
0 \\
0 \\
0
\end{array}
\right)
\end{equation*}
the intersection numbers $D_4 \cdot C_{0,3} = 1$, $D_0 \cdot C_{0,3} = 1$, $D_3 \cdot C_{0,3} = -1$, and $D_5 \cdot C_{0,3} = 1$, with all remaining intersection numbers equal to zero.

In particular, we obtain the coordinates for $C_{0,3}$ in the dual basis $N_1(Y_2)_\RR$: 
\begin{equation*}
[C_{0,3}] = -[D_3]^{\vee} + [D_4]^\vee + [D_5]^\vee = (-1,1,1).
\end{equation*}
Coordinates of the other generators of $N_1(Y_2)_\RR$ with respect to
the dual basis $\{ [D_3]^{\vee}, [D_4]^{\vee}, [D_5]^{\vee} \}$ are:
\begin{align*}
 &[C_{0,1}] = [C_{1,2} ] = [C_{1,4} ]= [C_{2,5} ] = (0,1,0), \\
 &[C_{0,2} ]  = (0,0,1) \\
 & [C_{2,4} ] = (0,-1,1)\\
 & [C_{0,3} ] = [C_{3,4} ] = [C_{3,5} ] =(-1,1,1), \\
 & [C_{0,5}] = [C_{4,5}] = (1,0,-1), \\
 &  [C_{0,4}] = (1,-1,0).
\end{align*}

By a \verb+PORTA+ calculation, the nef cone is
\begin{equation*}
\Nef(Y_2) = \langle [D_3] + [D_5], 2[D_3] + [D_4] + [D_5], [D_3] + [D_4] + [D_5] \rangle_{\geq 0}.
\end{equation*}
We denote the three extremal rays by $\eta_1$, $\eta_2$, and $\eta_3$, respectively. To calculate all pairs of intersections $\eta_i \cdot \eta_j$, we first calculate $[D_r] \cdot [D_s]$ for $r,s = 3,4,5$. For self-intersections, we rewrite the divisor using the relations in $N^1(Y_2)_\RR$ to make the intersection transverse. For example,

\begin{equation*}
[D_3]^2 = [D_3] \cdot([D_0] - [D_4] - [D_5])  = [C_{0,3}] - [C_{3,4}] - [C_{3,5}].
\end{equation*}
\end{example}
\noindent With respect to the dual basis $\{[D_3]^{\vee}, [D_4]^{\vee}, [D_5]^{\vee}\}$, we obtain $[D_3]^2 = (1, -1, -1)$. The other intersections are obtained analogously:
\begin{align*}
[D_4]^2 & = (1, -2, 0), \\
[D_5]^2 &= (1, -1, -1), \\
[D_3] \cdot [D_4] &= (-1, 1, 1), \\
[D_3] \cdot [D_5] &= (-1, 1, 1), \\
[D_4] \cdot [D_5] &= (1, 0, -1).
\end{align*}

Finally, we calculate the generators $\eta_i \cdot \eta_j$, $1 \leq i \leq j \leq 3$, in the dual basis $N_1(Y_2)_\RR$ by using the above intersections among the basis elements of $N^1(Y_2)_\RR$:
\begin{align*}
\eta_1^2 &= ([D_3] + [D_5])^2 = (0,0,0),\\
\eta_2^2 &= ( 2[D_3] + [D_4] + [D_5])^2 = (0,1,1),\\
\eta_3^2 &= ([D_3] + [D_4] + [D_5])^2  = (1,0,0),\\
\eta_1 \cdot \eta_2 &=( [D_3] + [D_5] ) \cdot ( 2[D_3] + [D_4] + [D_5] )  = (0,1,0),\\
\eta_1 \cdot \eta_3 &= ( [D_3] + [D_5] ) \cdot ([D_3] + [D_4] + [D_5])  = (0,1,0),\\
\eta_2 \cdot \eta_3 &= ( 2[D_3] + [D_4] + [D_5] ) \cdot([D_3] + [D_4] + [D_5])   = (0,1,1).
\end{align*}
Since the extremal ray $(0,0,1)$ of $\bmov(Y_2)$ does not appear among the generators of $\ci(Y_2)$, it follows that $\ci(Y_2) \subsetneq \bmov(Y_2)$.

To conclude this section, we use basic polyhedral geometry to give one example of how permutohedral spaces partially encode the geometry of $\M_{0,n}$. Namely, we show that extremal rays of the movable cone of $\LM_{n-2}$ pull back to extremal rays of the movable cone of $\M_{0,n}$. Recalling the Kapranov blow-up construction, we define $f: \M_{0,n} \to \LM_{n-2}$ to be the final (non-toric) composition of blow-ups.

\begin{prop}
\label{propCI:extRayMov}
Let $\gamma$ be an extremal ray of $\bmov(\LM_{n-2})$. Then $f^*(\gamma)$ is an extremal ray of $\bmov(\M_{0,n})$.
\end{prop}
\begin{proof}
Note that 
\begin{align*}
f^*: N_1(\LM_{n-2})_\RR &\to N_1(\M_{0,n})_\RR \text{ and} \\
f_*: N^1(\M_{0,n})_\RR &\to N^1(\LM_{n-2})_\RR
\end{align*}
are dual with respect
to the intersection pairing. Since $f_*$ is surjective, $f^*$ is
injective, and in particular $f^*(\gamma) \neq 0$. Moreover, by the
projection formula \ref{lemmaI:projection} with $[D] \in
\pEff(\M_{0,n})$,
\begin{equation*}
[D] \cdot f^*(\gamma) = f_*([D]) \cdot \gamma \geq 0,
\end{equation*}
since $f_*([D])$ is also an effective divisor class, hence
$f^*(\gamma) \in \bmov(\LM_{n-2})$.

Set $\rlm = \dim N_1(\LM_{n-2})_\RR$ and
$\rmn = \dim N_1(\M_{0,n})_\RR$. Extremality of
$\gamma$ implies that there exist $\rlm - 1$ linearly independent
defining hyperplanes of $\bmov(\LM_{n-2})$ intersecting $\gamma$ with
value zero, i.e. there exist linearly independent divisor classes $[D'_1],
\ldots, [D'_{\rlm - 1}]$ on $\LM_{n-2}$ satisfying $[D'_i] \cdot \gamma = 0$
for all $i$. Next select $\rlm - 1$ divisor classes
$[D_i] \in N^1(\M_{0,n})_\RR$ such that $f_*([D_i]) = [D_i]$ for all
$i$ (the $[D_i]$ are by construction linearly independent).

It is easy to see from the Kapranov blow-up construction that 
\begin{equation*}
\ker f_* = \langle [E_J]:  n-1 \in J\rangle,
\end{equation*}
while the projection formula implies that $[E_j] \cdot \gamma =0$ if
$n-1 \in J$.
Since the collection $\{[D_i]: i=1, \ldots, \rlm\} \cup \{[E_J]:
n-1 \in J\}$ is linearly dependent with cardinality $\rmn - 1$, it follows that $f^*(\gamma)$ is
an extremal ray of $\bmov(\M_{0,n})$.
\end{proof}

By applying this proposition to $\M_{0,6}$ and varying which marked
points are chosen as poles for $\LM_4$ (this choice is explained for
example in \cite{larsenThesis}, Sec. 3.3), we can obtain an enumeration of extremal rays common to $\bmov(\M_{0,6})$ and $\ci(\M_{0,6})$. This collection, however, does not give all common extremal rays: for example, the extremal ray 
\begin{equation}
\gamma = 6 [H]^\vee + 2 \sum_{i =1}^4 [E_i]^\vee + [E_{15}]^\vee + [E_{25}]^\vee + [E_{35}]^\vee, \nonumber
\end{equation}
and its symmetric analogues, is an extremal ray of both $\ci(\M_{0,6})$ and $\bmov(\M_{0,6})$, but it is not the pull-back of an extremal ray from $\bmov(\LM_4)$, as can be seen by examining the \verb+PORTA+ file for $\bmov(\LM_4)$.
 \section{Calculating complete intersection and movable cones}
 \label{secCI:intTh}

\begin{proof}[Proof of Proposition \ref {lemmaCI:nefCone}]
We present the proof only for $r=5$, since the other cases are standard (for details, see Chapter 4 of \cite{larsenThesis}).
These inequalities will follow by intersecting the divisor class $[D]$ with all classes of F-curves. Using the identification of the hyperplane class with the psi-class $\psi_6$ from \cite{MR1203685}, it is not hard to show that for each partition $\mu = (\mu_1, \mu_2, \mu_3, \mu_4)$,
\begin{equation*}
H \cdot F_\mu = 
\begin{cases}
1 & \text{ if }\mu_i = \{6\} \text{ for some } i, \\
0 & \text{else}.
\end{cases}
\end{equation*}

To intersect F-curves with the remaining elements of the Kapranov basis, we apply Proposition \ref{propI:FcurvePartition} and the dictionary between boundary and exceptional divisor classes from Equations (\ref{eq:dictionary}). The first set of inequalities arise from F-curves with partitions $(\mu_1, \mu_2, \mu_3, \mu_4)$ satisfying $|\mu_i| = 1$ for $1 \leq i \leq 3$, with $6 \in \mu_4$, while the second set of inequalities comes from such partitions with instead $6 \notin \mu_4$. Partitions with $|\mu_1| = |\mu_2| = 1$ and $|\mu_3| = |\mu_4| = 2$ such that $i \in \mu_3 \cup \mu_4$ give the third set of inequalities, while the final set results from such partitions when $6 \in \mu_1 \cup \mu_2$. 

Up to the action of the symmetric group permuting the four elements of the partition (which leaves the numerical class unchanged), the above partitions correspond to all possible partitions corresponding to F-curves in $\M_{0,6}$, so these inequalities define the nef cone. 
\end{proof}

\begin{proof}[Proof of Proposition \ref{lemmaCI:nefCurvesCone}]
Since we represent $[C] \in N_1(\M_{0,6})_\RR$ with respect to the dual Kapranov basis, by duality these inequalities can just be read off of the coordinates of the generators for $\pEff(\M_{0,6})$ expressed in the Kapranov basis for $N^1(\M_{0,6})_\RR$: the first set of inequalities are from intersecting $[C]$ with the $[E_i]$, the second from intersecting with the $[E_{ij}]$, the third from intersecting with $[\D_{ij}]$ where $6 \notin \{i,j\}$, and the last from intersecting with the Keel-Vermeire divisors.
\end{proof}

\begin{proof}[Proof of Proposition \ref{lemmaCI:doubleInt}]
The first equality of (\ref{eq:intZero}) holds since we can always choose a hyperplane not containing any of the points $p_i$. Since exceptional divisors corresponding to disjoint blow-up centers are also disjoint, the remaining equalities follow immediately, with the possible exception of the final one. To see that $E_{ij} \cap E_{ik} = \emptyset$ for $i,j,k$ distinct, let $\ell_{ij}, \ell_{ik} \subseteq \PP^3$ be the corresponding lines, and $p_i$ their intersection. Since the Kapranov construction requires blowing up in order of increasing dimension, after blowing up $p_i$ the proper transforms of $\ell_{ij}$ and $\ell_{ik}$ will be disjoint, giving last equality.

To express the intersections of divisor classes $[D'], [D'']$ in (\ref{eq:intDual}) in the dual basis, we intersect an arbitrary divisor class $[D] = d_h [H] + \sum_{i=1}^r [E_i] + \sum_{j,k = 1}^r [E_{jk}]$ with $[D'] \cdot [D'']$, giving an expression in the coefficients $d_h, d_i,$ and $d_{jk}$. Since the dual bases are related by the intersection product, this expression is the one-cycle $[D'] \cdot [D'']$ in the dual Kapranov basis once we substitute $[H]^\vee$ for $d_h$, $[E_i]^\vee$ for the $d_i$, and $[E_{jk}]^\vee$ for the $d_{jk}$. The triple intersection products required are standard calculations on $\M_{0,6}$ and toric varieties. For details, see Chapter 4 of \cite{larsenThesis}.

\end{proof}


\bibliographystyle{alpha}

\bibliography{../../../bibliography.bib}

\end{document}